\documentclass[final,12pt]{elsarticle}
\usepackage{xcolor,soul,cancel}
\usepackage[color]{showkeys}
\usepackage{amssymb}
\usepackage{amsmath}
\usepackage{cases} 
\usepackage{amsthm}
\usepackage{graphicx}
\usepackage{float}
\usepackage{enumitem}
\newlist{alphalist}{enumerate}{1} 
\setlist[alphalist]{label=(\alph*)} 
\journal{arXiv}

\definecolor{refkey}{rgb}{0,1,1}
\definecolor{labelkey}{rgb}{1,0,0}

\numberwithin{equation}{section}

\newtheorem{thm}{Theorem}[section]
\newtheorem{lem}[thm]{Lemma}
\newtheorem{cor}[thm]{Corollary}
\newtheorem{prop}[thm]{Proposition}
\newtheorem{rem}[thm]{Remark}
\newtheorem{ex}[thm]{Example}
\newtheorem{defi}[thm]{Definition}

\newcommand{\abs}[1]{\left\vert#1\right\vert}
\newcommand{\rk}{\operatorname{rank}}

\newcommand{\Span}{\operatorname{Span}}
\newcommand{\re}{\operatorname{Re}} 
\newcommand{\im}{\operatorname{Im}}

\newcommand{\eq} [1] {\begin{equation}\label{#1}\quad}
\newcommand{\en} {\end{equation}}
\newcommand{\scal}[1]{\langle#1\rangle}
\newcommand{\norm}[1]{\left\Vert#1\right\Vert}
\newcommand{\C}{\mathbb C}\newcommand{\R}{\mathbb R}

\newcommand{\diag}{\operatorname{diag}}
\newcommand{\conv}{\operatorname{conv}}
\newcommand{\defe}{\operatorname{def}}

\theoremstyle{definition}

\theoremstyle{definition}

\makeatletter
\renewcommand*\env@matrix[1][*\c@MaxMatrixCols c]{%
  \hskip -\arraycolsep
  \let\@ifnextchar\new@ifnextchar
  \array{#1}}
\makeatother

\begin{document}
\begin{frontmatter}

\title{ On Kippenhahn curves of low rank partial isometries
\tnoteref{support}}

\author[nyuad]{Nikita Popov}
\ead{np2361@nyu.edu}

\author[eric]{Eric Shen}
\ead{erick.2013@yandex.ru}
\author[nyuad]{Ilya M. Spitkovsky}
\ead{ims2@nyu.edu, ilya@math.wm.edu, imspitkovsky@gmail.com}
\address[nyuad]{Division of Science and Mathematics,
New York  University Abu Dhabi (NYUAD), Saadiyat Island,
P.O. Box 129188 Abu Dhabi, United Arab Emirates}
\address[eric]{Moscow State University, Moscow, 119991, Russia }

\tnotetext[support]{The results are partially based on the Capstone project of [NP] (2023-24 academic year) under the supervision of [IMS]. The latter was also
supported in part by Faculty Research funding from the Division of Science and Mathematics, New York University Abu Dhabi. }

\begin{abstract} Conditions are established for rank three partial isometries to have circular components contained in their Kippenhahn curves. In particular, such matrices with circular numerical ranges are described. It is also established that the Gau-Wang-Wu conjecture holds for matrices under consideration.
\end{abstract}

\end{frontmatter}

\section{Introduction}

Let $A$ be a bounded linear operator acting on a Hilbert space $\mathcal H$. Its {\em numerical range} $W(A)$ is defined as
\eq{nr}
W(A)=\{\scal{Ax,x}\colon x\in \mathcal H, \norm{x}=1\}, 
\en 
where $\norm{.}$ is the norm associated with the scalar product $\scal{.,.}$ on $\mathcal H$. For any operator $A$, the set $W(A)$ is convex according to the Toeplitz-Hausdorff theorem; see, e.g. \cite{GauWu} for this and other known properties of numerical ranges. 

In this paper, we restrict our attention to {\em finite rank} operators, i.e., operators $A$ with the range $\mathcal R(A)$ having dimension $k<\infty$. The following simple observation (used in particular in \cite[Section 3]{RS13}) is useful in this setting.
\begin{prop}\label{th:fr}
For a rank $k$ operator $A$ \[ \mathcal M:=\mathcal R(A)+\mathcal R(A^*)\]  is a reducing subspace of dimension $n\leq 2k$, and the restriction $A|\mathcal M$ of $A$ onto $\mathcal M^\perp=\ker A\cap\ker A^*$ is the zero operator. 
\end{prop} 
So, it suffices to consider $A|\mathcal M$. Slightly abusing the notation, we will identify it with its matrix representation while denoting the latter also by $A\,(\in\C^{n\times n})$. 

In this setting $W(A)$ coincides with the convex hull of a certain algebraic curve $C(A)$:
\eq{nuco} W(A)=\conv C(A). \en  
This curve is obtained as the envelope of the family
\eq{env} \{e^{-i\theta}(\lambda_j(\theta)+i\R)\colon \theta\in (-\pi,\pi],\ j=1,\ldots,n\}, \en 
where $\lambda_j(\theta)$ are the roots of \eq{chpo} P_A(\lambda,\theta)=\det(\re(e^{i\theta}A)-\lambda I)\en  
--- the characteristic polynomial of the hermitian part of $e^{i\theta}A$, also called the {\em Kippenhahn polynomial} of $A$. Moreover, the spectrum $\sigma(A)$ of $A$ coincides with the set of the foci of $C(A)$.

\noindent
These observations go back to \cite{Ki}, see also the English translation \cite{Ki08}.

We are interested in conditions sufficient for $C(A)$ to contain a circle $\mathcal C$; in particular, for $W(A)$ to be a circular disk. The latter happens, of course, if and only if $\mathcal C$ is the boundary $\partial W(A)$ of $W(A)$, and so all other components of $C(A)$ lie inside $\mathcal C$. 

Some general observations concerning this phenomenon are gathered (mainly, for convenience of reference) in the next Section~\ref{s:prel}.

 In Section~\ref{s:pifk} we restrict our attention further to the so called partial isometries and dwell on the conditions for their Kippenhahn curves to contain a circle of radius $1/2$ centered at the origin. Section~\ref{s:gww3} is devoted to the Gau-Wang-Wu conjecture according to which the numerical range $W(A)$ of a partial isometry $A$ has to be centered at the origin, provided that it is assumed to have a circular shape. We prove this conjecture for $A$ having rank three. Actually, we establish a more precise statement that if $C(A)$ for such $A$ contains a circle, then this circle has to be centered at the origin. 

Sections~\ref{s:circgen}--\ref{s:r3d2} are devoted to a complete description  of rank three partial isometries with Kippenhahn curves containing a circular component. Section~\ref{s:circgen} contains a general criterion and the result for matrices of defect one. The criterion is specified and simplified for matrices of defect three (i.e., nilpotent) in Section~\ref{s:nilp}. We revisit the case of circles having radius $1/2$ in Section~\ref{s:C1/2}. The remaining, most complicated, situation -- matrices of defect two and circles of radii different from $1/2$ -- is treated in Section~\ref{s:r3d2}. Throughout Sections~\ref{s:circgen}--\ref{s:r3d2}, criteria for $W(A)$ to be circular disks are singled out. The final Section~\ref{s:Ex} contains figures and numerical examples illustrating the results of Section~\ref{s:r3d2}. 

\section{Preliminary results}\label{s:prel} 
Since $W(A)$, and even $C(A)$, are invariant under unitary similarities, we may without loss of generality substitute $A$ by any matrix unitarily similar to it. In particular, for $A$ with an $m$-dimensional kernel we may without loss of generality suppose that it has the form 
\eq{BC} \begin{bmatrix} 0  & B \\ 0 & C\end{bmatrix}, \en
with the $m$-by-$m$ upper left block. Moreover, when it is convenient, due to the Schur lemma we may suppose that $C$ is upper triangular. 

 Representation \eqref{BC} was used repeatedly in \cite{GWW}. In particular, the following criterion was derived there:
\begin{prop}\label{th:uirr}{\em \cite[Proposition 2.6]{GWW}.} The matrix \eqref{BC} is unitarily irreducible if and only if so is its block $C$ while the block $B$ has full rank.\end{prop} 

For convenience of reference, we will denote by $\mathcal C_{a,r}$ the circle centered at $a$ of radius $r$, abbreviating this simply to $\mathcal C_r$ when $a=0$. We will silently suppose that $r>0$ unless explicitly stated otherwise, in which case we will use the convention $\mathcal C_0=\{0\}$. 

\begin{prop}\label{th:critcirc}
The Kippenhahn curve $C(A)$ of a matrix $A$ contains $\mathcal C_{r}$ if and only  if the Kippenhahn polynomial \eqref{chpo} is divisible by $\lambda^2-r^2$.

\end{prop} 
This is a straightforward consequence of $C(A)$ being the envelope of \eqref{env} and was used repeatedly in the literature (see, e.g., Fact 1 in \cite{HeSpitSu}).
\begin{thm}\label{th:multeig}
If the Kippenhahn curve $C(A)$ of a matrix $A\in\C^{n\times n}$ contains $\mathcal C_{a,r}$, then  $a$ is a defective eigenvalue of $A$.
\end{thm} 
\begin{proof}Suppose $C(A)\supset  \mathcal C_{a,r}$.Then $a$, being a focus of $C(A)$, is an eigenvalue of $A$. Considering $A-aI$ in place of $A$, we may without loss of generality suppose that $a$=0. Denoting by $m$ the geometric multiplicity of this eigenvalue, let $A$ be
as in \eqref{BC}. 

According to Proposition~\ref{th:critcirc}, 
\[ \det\begin{bmatrix} -2rI & \omega B\\ \overline{\omega}B^* & \omega C+\overline{\omega}C^*-2rI\end{bmatrix}=0, \quad \abs{\omega}=1,  \]
or, using the Schur complement: 
\eq{Schur} \det (\omega C +\overline{\omega}C^*+H)=0, \quad \abs{\omega}=1, \en
where \eq{Hr} H=B^*B/2r-2rI.\en 

Rewriting \eqref{Schur} as 
\[ \det ( zC +z^{-1}C^*+H)=0, \quad \abs{z} =1, \] 
observe that the determinant under consideration, being a polynomial in $z$ and $z^{-1}$, is identically zero for all $z\neq 0$, not only for unimodular $z$. But then, multiplying by $z$:
\[ \det ( z^2C +zH+C^*)=0, \quad z\neq 0, \]
and so, by continuity, $\det C (=\overline{\det C^*})=0$. The algebraic multiplicity of zero as an eigenvalue of the matrix \eqref{BC} is therefore at least $m+1$.
\end{proof} 
A slightly stronger result holds: since $\deg\det(z^2C-zH+C^*)\leq 2(n-m)$, it suffices to suppose that $C(A)$ contains $2(n-m)+1$ points of a circle $C$ to conclude that, in fact, $C\subset C(A)$ (and, consequently, the center of $C$ is a defective eigenvalue of $A$). 

This observation provides an alternative proof of the following result, which is a restatement of \cite[Theorem 1]{Wu11}.
\begin{thm}\label{th:Wu}
Let $A\in\C^{n\times n}$ be such that $W(A)$ lies in the closed circular disk 
$\mathcal D:=\{z\colon \abs{z-a}\leq r\}$  while its boundary  $\mathcal C_{a,r}=\mathcal C$ contains at least $n-m+1$ points of $W(A)$.  (Here $m\geq 0$ is the geometric multiplicity of $a$ as an eigenvalue of $A$.) Then, in fact, $W(A)=\mathcal D$, $m>0$, and $a$ is a defective eigenvalue of $A$.    
\end{thm}
\begin{proof}
Let $\mathcal C\cap\sigma(A)=\{z_1,\ldots,z_k\}$. Then by the Donoghue theorem $A$ is unitarily similar to a block diagonal matrix $\diag[A_1,A_2]$, where $A_1$ is normal, with the spectrum $\{z_1,\ldots,z_k\}$, and  $A_2\in\C^{(n-k)\times (n-k)}$ is such tat $W(A_2)\subseteq\mathcal D$. Moreover, $W(A_2)\cap \mathcal C$ contains at least $(n-k)-m$ points, all of which are tangent points of $\mathcal C$ and $C(A_2)$. Since these points are counted with their multiplicities, the observation made right before the statement of this theorem, with $A$ replaced by $A_2$ and $n$ with $n-k$, implies that $W(A_2)=\mathcal D$,  $a$ is a defective eigenvalue of $A_2$, and therefore of $A$. Furthermore, $W(A_1)=\conv\{z_1,\ldots,z_k\}\subset\mathcal D$, and so $W(A)=W(A_2)=\mathcal D$. 
\end{proof}
\noindent Let us also mention a more recent related result \cite[Theorem 4.1]{PSW}. 

To conclude this auxiliary section, for convenience of reference, let us state a useful result on singular polynomial matrix functions. 
\begin{prop} \label{prop:polynomial_solution}
    Let $P(z) = \sum_{i} P_iz^i$ be a matrix polynomial over $\mathbb{C}$, such that $\det P(z) = 0$ identically. Then there exists a vector polynomial $v(z)$ such that $P(z)v(z) = 0$, also  identically.
\end{prop}
We did not find this exact statement in the literature, but it is of course known and follows immediately from the normal Smith form (see, e.g., \cite{Gan1}, Theorem 3 of Chapter 6) of $P$.  

\section{Finite rank partial isometries}\label{s:pifk} 
In what follows, we restrict our attention to {\em partial isometries}, i.e., $A$ satisfying 
\eq{pi} \norm{Ax}=\norm{x} \text{ for all } x\perp\ker A. \en 

A partial isometry with a trivial kernel is nothing but a unitary matrix, and so this extreme case is not of particular interest. In other cases, we will be using form \eqref{BC} of $A$. Note that then \eqref{pi} can be rewritten as \eq{BC1} B^*B+C^*C=I. \en 

The simplest ``true'' partial isometry is the Jordan block
$J_2=\left[\begin{smallmatrix}0 & 1\\ 0 & 0\end{smallmatrix}\right]$. It is well known (and easy to check) that $C(J_2) =\mathcal C_{1/2}$. 
 Consequently, $W(J_2)$ is the circular disk $\{ z\colon \abs{z}\leq 1/2\}$. 

It is therefore clear that for any matrix $A$ unitarily similar to a block diagonal matrix with $J_2$ as one of its blocks, \eq{cont1/2} C(A)\supset \mathcal C_{1/2}. \en 
We will now show that for partial isometries, and under some additional conditions, the converse is also true.

\begin{thm} \label{thm:mat_poly_general}
Let $A$ be a rank $k$ partial isometry such that \eqref{cont1/2} holds.  If in its representation \eqref{BC} \eq{k/2cond} \dim\ker C\geq\lfloor\frac{k}{2}\rfloor,\en then $A$ is unitarily reducible with one of the blocks equal $J_2$.
\end{thm}
\begin{proof}
It suffices to show that \eq{nz} \ker C\cap\ker C^*\neq\{0\}.\en   Indeed, if this is the case, we can use a block diagonal unitary similarity to simultaneously replace $C$ and $B$ with $[0]\oplus C_0$ and $[1]\oplus B_0$, respectively.

Observe that \eqref{nz} follows immediately from dimensionality considerations provided that $\dim\ker C>\lfloor\frac{k}{2}\rfloor$. So, only the case $\dim\ker C=\lfloor\frac{k}{2}\rfloor$ needs to be considered.

From the proof of Theorem~\ref{th:multeig} one can see that \eqref{cont1/2} holds if and only if $\det (z^2C+zH+C^*)$ is identically zero. Here $H$ is given by \eqref{Hr} with $r=1/2$. In other words, $H=B^*B-I$. According to \eqref{BC1}, this simplifies further to $H=-C^*C$, and so
\[ \det P(z)=0, \quad z\in\C,\quad \text{ where } P(z)= z^2C-zC^*C+C^*. \]

According to Proposition~\ref{prop:polynomial_solution}, there exist vector polynomials $v(z) = v_dz^d + \ldots + v_1z+v_0$, such that $P(z)v(z)=0$ identically. Choosing such a polynomial with minimal possible $d$, we have, in particular, $v_0, v_d \neq 0$. If $d=0$, it is at hand that $Cv_0=C^*v_0=0$, and thus $\ker C\cap\ker C^* \neq \{0\}$. Assume $d > 0$ while $\ker C\cap\ker C^*=\{0\}$. In what follows we assume that vectors $v_i$ with  $i>d$ or $i < 0$ are zero.

{\sl Case 1.} Rank $A$ is even. Then $\ker C\dotplus \ker C^*$ is the entire underlying $k$-dimensional space. 
     
Let us look at the three lowest power terms of $P(z)v(z)$:
         \begin{enumerate}
        \item $z^0: C^*v_0=0$
        \item $z^1:C^*v_1-C^*Cv_0=0 $
        \item $z^2: C^*v_2-C^*Cv_1 + Cv_0=0.$
        \end{enumerate}
    From the $z^0$ term we get $C^*v_0 = 0$, from the $z^2$ term we get $$Cv_0 = C^*(Cv_1 - v_2) \in \im(C) \cap \im(C^*) = (\ker(C) \dotplus \ker(C^*))^\perp = \left\{0\right\}.$$
    Hence $Cv_0 = C^*v_0 = 0$, which is a contradiction with $v_0 \neq 0$.

{\sl Case 2.} Rank $A$ is odd. Then $\ker C\dotplus\ker C^*$ has codimension one, with 
$\mathcal L:=\im C\cap\im C^*$ being its (one-dimensional) orthogonal complement. The underlying space therefore partitions as $(\ker C\dotplus\ker C^*)\oplus \mathcal L$. 
    
Let us consider the  four highest terms of $P(z) v(z)$:
    \begin{enumerate}
        \item $z^{d+2}: Cv_d = 0$
        \item $z^{d+1}: Cv_{d-1} - C^*Cv_d = 0$
        \item $z^d\quad: C v_{d-2} - C^*Cv_{d-1} + C^*v_d = 0$
        \item $z^{d-1}: C v_{d-3} - C^*Cv_{d-2} + C^*v_{d-1} = 0.$
    \end{enumerate}
From the first two we conclude that $v_d$ and $v_{d-1}$ belong to $\ker C$. Hence, the third equation simplifies to 
\eq{cvd} C v_{d-2}  + C^*v_d = 0,\en 
and so $Cv_{d-2}$ and $C^*v_d$ both belong to $\mathcal L$. Note that these vectors are non-zero, since otherwise $v_d$ would belong to  $\ker C\cap\ker C^*$. 

Due to \eqref{cvd}, the fourth equation can be rewritten as 
\[  C v_{d-3} + C^*(C^*v_{d} + v_{d-1}) = 0, \]
implying that $C^*(C^*v_{d} + v_{d-1})\in\mathcal L$. Since $\mathcal L$ is one-dimensional, this vector is therefore collinear with $C^*v_d$:
\[ C^*(C^*v_{d} + v_{d-1})-\lambda C^*v_d=0 \text{ for some } \lambda\in \C. \]
Equivalently, $C^*v_d+v_{d-1}-\lambda v_d\in\ker C^*$. On the other hand, 
$C^*v_d+v_{d-1}-\lambda v_d\in\ker C\oplus\mathcal L$, which has only a trivial intersection with $\ker C^*$. 

Consequently, $C^*v_d+v_{d-1}-\lambda v_d=0$, which implies that $C^*v_d\in\ker C$ along with $v_d, v_{d-1}$.
Recalling that $C^*v_d\in\mathcal L$, this is only possible if $C^*v_d=0$ --- a contradiction.  
\end{proof} 

Due to Theorem~\ref{th:multeig}, condition \eqref{k/2cond} is satisfied automatically when $k\leq 3$. So, the following statement holds.
\begin{cor} \label{th:1/2}
Let $A$ be a partial isometry of rank at most three. Then $C(A)$ contains the circle $\mathcal C_{1/2}$ if and only if $A$ is unitarily reducible, with one of the blocks equal $J_2$.\end{cor}

Note that any matrix $A$ of norm one, not necessarily a partial isometry, is unitarly reducible with one of the blocks equal $J_2$, provided that the numerical radius $w(A)$ of $A$ equals $1/2$, see \cite[Lemma 2]{JL}. For partial isometries, a repeated application of this result yields the following: 
\begin{thm}\label{th:cit1/2} Let $A\in\C^{n\times n}$ be a partial isometry. Then $w(A)=1/2$ if and only if $A$ is unitarily similar to the direct sum of $m (\in\{1,\ldots\lfloor n/2\rfloor\})$ blocks $J_2$ with an $(n-2m)$-by-$(n-2m)$ zero block. If this is the case, then $W(A)=\{ z\colon \abs{z}\leq 1/2\}$ while $C(A)$ consists of $m$ copies of the circle $\{z\colon\abs{z}=1/2\}$ and $n-2m$ copies of the origin. \end{thm}


The bound in Theorem~\ref{thm:mat_poly_general} is sharp. To illustrate, consider the following two families of examples.

\begin{ex} \label{ex:odd_dim}
Let 
    $$C = \begin{bmatrix}
    1 & 0 & 0 & 0 & 0 \\
    0 & 0 & 1 & 0 & 0\\
    0 & 0 & 0 & 0 & 1\\
    0 & 1 & 0 & 0 & 0 \\
    0 & 0 & 0 & 0 & 0 \end{bmatrix}.$$
Then  $\ker C = \Span(e_4) \text{ while } \ker C^*  = \Span(e_5)$. Observe that \eq{Cz} z^2C-zC^*C+C^* = \begin{bmatrix}
z^2 - z + 1 & 0 & 0 & 0 & 0 \\
0 & -z & z^2 & 1 & 0 \\
0 & 1 & -z & 0 & z^2 \\
0 & z^2 & 0 & 0 & 0 \\
0 & 0 & 1 & 0 & -z \\
\end{bmatrix} \en 
is rank-deficient. Taking direct sum of $C$ with $m$ copies of $2\times 2$  identity matrices $I_2$ and $l$ copies of $J_2$ ($m,l\geq 0$), we obtain a $k$-by-$k$ matrix $$C_k = C \oplus I_2\oplus \ldots \oplus I_2 \oplus J_2 \oplus\ldots \oplus J_2$$ with $\lfloor k/2\rfloor=m+l+2$ and $\dim\ker C_k=l+1$. So, $\lfloor k/2\rfloor-\dim\ker C_k=m+1$, and by choosing $l$ appropriately we can make this difference attain all natural values between one and $\lfloor k/2\rfloor-1$. At the same time, it is easy to see that \[ \ker C_{k} \cap \ker C_{k}^* = \ker C \cap \ker C =\left\{0\right\},\] while $z^2C_{k}+C^*_{k}-zC_k^*C_k$ is singular along with \eqref{Cz}.
\end{ex}

\begin{ex}\label{ex:even_dim}
To obtain an even-dimensional family, one can start with the following $4\times4$ matrix $C$:
$$C = \begin{bmatrix}
    0 & 0 & 1 & 0 \\
    0 & 0 & 0 & 0 \\
    0 & 1 & 0 & 0 \\
    1 & 0 & 0 & 0
\end{bmatrix}$$    
and define 

$$C_{k} = C \oplus I_2\oplus \ldots \oplus I_2 \oplus J_2 \oplus\ldots \oplus J_2$$
analogously to Example~\ref{ex:odd_dim}.

Indeed, observe that $\ker C = \Span(e_4)$ and $\ker C^* = \Span(e_2)$, so $\ker(C) \cap \ker(C^*) = \{0\}$.
On the other hand, $\rk (z^2C-zC^*C+C^*) \leq 3$. Indeed,
$C^*C = \diag [1,1,1,0]$, and 
$$z^2C-zC^*C+C^* = \begin{bmatrix}
    -z & 0 & z^2 & 1 \\
    0 & -z & 1 & 0 \\
    1 & z^2 & -z & 0 \\
    z^2 & 0 & 0 & 0
\end{bmatrix}$$
has the determinant identically equal to zero.
\end{ex}

Examples~1 and 2 show that condition $\det(z^2C-zC^*C+C^*)=0$ may hold identically in $z$ while the  intersection $\ker C \cap \ker C^*$ is trivial. Curiously, the situation change if we flip two coeficients of the polynomial under consideration. Namely: 
\begin{prop} \label{prop:new_poly}
If for every $z\in\C$  \eq{new_poly} \det(z^2C+zC^*-C^*C)=0,  \en
then the intersection $\ker C\cap \ker C^*$ is non-trivial.
\end{prop}
\begin{proof}
According to Proposition~\ref{prop:polynomial_solution}, \eqref{new_poly} implies the existence of vector polynomials $v(z) = v_dz^d + \ldots + v_1z+v_0$, such that 
\eq{iden1} (z^2C-C^*C+zC^*)v(z)\equiv 0. \en  Let us choose $v$ of the smallest possible degree $d$; then, in particular, $v_0, v_d\neq 0$.
Let us look at the four lowest terms of \eqref{iden1}. 
    \begin{enumerate}
        \item $z^0: C^*Cv_0=0$
        \item $z^1:C^*v_0-C^*Cv_1=0 $
        \item $z^2: Cv_0+C^*v_1 - C^*Cv_2=0$
        \item $z^3: Cv_1 + C^*v_2 - C^*Cv_3 = 0$
        \end{enumerate}
(Here, similarly to the proof of Theorem~\ref{thm:mat_poly_general}, by convention $v_j=0$ for $j>d$.)

From $z^0$ equation we obtain $Cv_0 = 0$, from $z^1$ equation we infer that \eq{cv10} Cv_1-v_0\in \ker C^*.\en  
If $d=0$, then $v_1=0$ by convention, implying $v_0\in\ker C^*$, and we are done. 

Let now $d>0$. Then $z^1$ equation implies $Cv_1\in \ker C + \ker C^*$, or, equivalently, $Cv_1\in (\im C\cap\im C^*)^\perp$. 
On the other hand,  from the $z^3$ equation we see that $Cv_1 \in \im C^*$. Since obviously $Cv_1\in\im C$, this is only possible if $Cv_1 = 0$. So, \eqref{cv10} implies $C^*v_0=0$, as in the case $d=0$. This completes the proof. 
\end{proof}

\section{Gau-Wang-Wu's conjecture for rank three operators} \label{s:gww3}
It was conjectured in Gau-Wang-Wu's paper \cite{GWW}  that, if $A\in\C^{n\times n}$ is a partial isometry with a circular numerical range $W(A)$, then the latter is necessarily centered at the origin. For brevity, we will refer to this statement as the GWW conjecture. 

The conjecture was proved in the same paper \cite{GWW} for $n\leq 4$. Due to Proposition~\ref{th:fr}, this implies the validity of the GWW conjecture for partial isometries of ranks one and two. The statement still holds for $n=5$ (see \cite{SuSWe}), as well as for any partial isometry with a one-dimensional kernel \cite{WeSp}, where the proof involves using elliptic functions; see \cite{AdCoGo25} for a more direct proof. So, the simplest case not covered by the already established results is that of the 6-by-6 partial isometries of rank three. We tackle it in this section. 

Before we begin, observe that the class of partial isometries is invariant under unitary similarities and rotations. These operations also are rank-preserving, and do not change the shape of $C(A)$, including the property of its circular component, if any, to be centered at the origin. Therefore, the following equivalence relation becomes handy.

\begin{defi}
The matrices $A$ and $B$ are {\em circularly equivalent} (denoted $A \sim B$) if there exist a unitary matrix $U$ and a real number $\theta$ such that $B = e^{i\theta}UAU^*$.
\end{defi}

\begin{prop}\label{th:cnform}A rank three $6$-by-$6$ partial isometry is circularly equivalent to the matrix
\eq{6x6canform} 
\begin{bmatrix}[c c c | c c c]  0 & 0 & 0 & \sqrt{1-a^2} & ba & da\\ 0 & 0 & 0 & 0 & v & 0\\ 0 & 0 & 0 & 0 & c & e\\ \hline 0 & 0 & 0 & a & -b\sqrt{1-a^2} & -d\sqrt{1-a^2}\\ 0 & 0 & 0 & 0 &\lambda_2 & f\\ 0 & 0 & 0 & 0 & 0 & \lambda_3 \end{bmatrix},
\en 
where $a,b,c,v\geq 0$, 
\eq{ortho}\begin{array}{c}  b d+ce+f\overline{\lambda_2}=0, \\
b^2+c^2+v^2+\abs{\lambda_2}^2=1, \\ \abs{d}^2+\abs{e}^2+\abs{f}^2+\abs{\lambda_3}^2= 1. \end{array} 
\en 
\end{prop} 
\begin{proof} We may suppose that $A$ is of the form \eqref{BC}, with an upper triangular $3$-by-$3$ block $C$. Via an appropriate rotation, one of its eigenvalues can be made non-negative; this is the element $a$ in \eqref{6x6canform}. A unitary similarity of the form $U\oplus I_3$ can be used to nullify the $(2,4), (3,4)$ and $(2,6)$ entries of $A$ without changing $C$. Yet another unitary similarity, this time diagonal, can be used to make the $(1,4), (1,5), (1,6), (2,5)$ and $(3,5)$ entries of $A$ non-negative. The special form of the $(1,4), (1,5), (1,6), (4,5), (4,6)$ entries then follows from the fact that the fourth column of $A$ has length one and also is orthogonal to the fifth and sixth columns. In turn, conditions \eqref{ortho} just mean that the columns 5,6 have unit length and orthogonal to each other. \end{proof} 

\begin{thm}\label{th:GWW63}Let $A$ be a $6$-by-$6$ rank three partial isometry. If $C(A)$ contains a circle, the latter is centered at the origin.\end{thm} 
\begin{proof}Suppose otherwise, i.e., $C(A)$ contains a circle centered at $a\neq 0$.  Without loss of generality, $a>0$ (which can be achieved by an appropriate rotation of $A$),  the matrix $A$ equals  \eqref{6x6canform} and, according to Theorem~\ref{th:multeig}, $\lambda_2=a$. Relabeling $\lambda_3=g$, one gets
\[ A-aI= \begin{bmatrix}[c c c | c c c]  -a & 0 & 0 & \sqrt{1-a^2} & ba & da\\ 0 & -a & 0 & 0 & v & 0\\ 0 & 0 & -a & 0 & c & e\\ \hline 0 & 0 & 0 & 0 & -b\sqrt{1-a^2} & -d\sqrt{1-a^2}\\ 0 & 0 & 0 & 0 & 0 & f\\ 0 & 0 & 0 & 0 & 0 & g-a \end{bmatrix}. \]
Consider the Kippenhahn polynomial of $X = A - aI$:
$$P_X(\theta, \lambda) = \sum_{i=1}^4\sum_{j=0}^1 \cos^i(\theta)\sin^j(\theta) A_{ij}(\lambda),$$
where $A_{ij}(\lambda)$ do not depend on $\theta$. Observe that the functions $\cos^i(\theta) \sin^j(\theta)$ for $i = 0, 1, 2, 3,4 $ and $j = 0, 1$ are linearly independent. 
Therefore, if $\lambda= r$ is the radius of the circle contained in $C(A)$, then $A_{ij}(r) = 0$ for all $i,j$. In particular, $A_{30}=A_{31}=A_{40}=0$. On the other hand, direct computations show that 
$A_{31}=\im g, A_{40}=a-\re g$. So, $g=a >0$. With this in mind, 
$A_{30}$ simplifies to
$$A_{30}= r^2a - \frac{1}{4}[a|f|^2 + b\re(\overline{d}f ) + c\re(\overline{e}f )].$$
But the first equality in \eqref{ortho}, with $\lambda_2=a$ taken into consideration, implies that 
$a\abs{f}^2 + b\re(\overline{d}f) + c\re(\overline{e}f)=0$. So, $A_{30} = r^2a = 0$.
As we assumed $r > 0$, it follows that $a=0$ --- a contradiction.
\end{proof}
When combined with Proposition~\ref{th:fr}, Theorem~\ref{th:GWW63} implies that the GWW conjecture holds for partial isometries of rank three independent of their size. On the other hand, the rank four case, while covered for $n=5$ by the results of \cite{SuSWe, WeSp}, remains open starting with $n=6$.

\section{Kippenhahn curves containing circular components} \label{s:circgen}
Here we give a complete characterization of the rank three partial isometries with Kippenhahn curves containing circular components. Our main tool is Proposition~\ref{th:critcirc}. Observe therefore that in our setting the Kippenhahn polynomial \eqref{chpo} of $A$ has the form
$$ P_A(\theta, \lambda) = \sum_{i=0}^2\sum_{j=0}^2 \cos^i(\theta) \sin^j (\theta)A_{ij},$$
where the coefficients $A_{ij}$ do not depend on $\theta$. Formulas for $A_{ij}$ (obtained with the use of SageMath) are rather cumbersome, and we will be providing their explicit form only as needed. 
\begin{prop} \label{two_eigenvalues_case_pro}  Assume that $A\in \C^{6\times 6}$ is a rank-three  partial isometry such that $C(A)$ contains a circle. Then 
\eq{6x60}A \sim \begin{bmatrix}
0 & 0 & 0 & 1 & 0 & 0\\
0 & 0 & 0 & 0 & b & 0\\
0 & 0 & 0 & 0 & c & e\\
0 & 0 & 0 & 0 & d & f\\
0 & 0 & 0 & 0 & a & g\\
0 & 0 & 0 & 0 & 0 & h
\end{bmatrix}, \en
where $a,b,c,d,e\geq 0$, and $ah=0$ or $d=f=0$. 
\end{prop}
Note that the ``or'' above is non-exclusive. 
\begin{proof}
Suppose that $C(A)$ contains a circle. By Theorem~\ref{th:GWW63}, this circle is centered at the origin, and by Theorem~\ref{th:multeig}, zero is then a defective eigenvalue of $A$. In other words, the block $C$ of its representation \eqref{BC} is singular. Consequently (with a slight abuse of notation), \eqref{6x6canform} becomes \eqref{6x60}. As in the proof of Proposition~\ref{th:cnform}, the non-negativity of $a$ may be achieved by a rotation, and that of $b,c,d,e$ --- by an appropriate diagonal similarity. 

Now, let the circle contained in $C(A)$ have radius $r >0$. Then all $A_{ij}(r)$ should equal zero by Proposition~\ref{th:critcirc}. Direct (computer assisted) computations show that, in particular, 
\[ A_{20} = ar^2\re(h)(r^2-1/4), \quad A_{11} = ar^2\im(h)(r^2-1/4). \]
So, $ah(r^2-1/4)=0$, and thus, either $ah=0$ or $r=1/2$. 

In the latter case invoke Theorem~\ref{thm:mat_poly_general} (or rather its proof), according to which for rank $k$ matrices $A$ satisfying \eqref{cont1/2},  \eqref{k/2cond} containment \eqref{BC1} is only possible if \eqref{nz} holds. Since $k=3$ and $C$ is singular, condition \eqref{k/2cond} holds automatically. If $ah\neq 0$, then $\ker C$ is the span of the first coordinate vector and, in order for it to have a non-trivial intersection (in this particular case, same as to coincide) with $\ker C^*$, it is necessary and sufficient that $d=f=0$.
\end{proof}
In what follows, it is convenient to distinguish the cases based on the defect of the zero eigenvalue of $A$. For convenience of reference, we will denote it by $\defe A$ and, for brevity,  call it the {\em defect of} $A$. 

According to Theorem~\ref{th:multeig}, $C(A)$ contains no circles if $\defe A=0$, i.e., zero is a regular eigenvalue of $A$. So, let us move to the case $\defe A=1$. 
\begin{thm}\label{th:def1}
Let $A$ be a rank three partial isometry with $\defe A=1$. Then $C(A)$ contains a circle if and only if $A$ is unitarily reducible, with one of the blocks equal $J_2$. If this is the case, then this circle has radius $1/2$ and there are no other circles contained in $C(A)$. In particular, $W(A)$ is not a circular disk.\end{thm}
\begin{proof}
Sufficiency is obvious, while necessity follows from Proposition~\ref{two_eigenvalues_case_pro}. Indeed, $\defe A=1$ if and only if in \eqref{6x60} $ah\neq 0$.

Furthermore, if $A$ is unitarily similar to $J_2\oplus B$, then $\defe B=0$, so there are no circles in $C(B)$. Consequently,  $\mathcal C_{1/2}$ is the only circle in $C(A)$. 

Finally, suppose $W(A)$ is a circular disk. Then its boundary should coincide with the unique circular component of $C(A)$. In other words, we should have $w(A)=1/2$. But this is in contradiction with Theorem~\ref{th:cit1/2}.
\end{proof}

Let now $\defe A\geq 2$, i.e. $ah=0$ in its canonical form \eqref{6x60}. Without loss of generality, we may suppose that $a=0$, so that $A$ is circularly equivalent to 
\eq{6x61} \begin{bmatrix}
0 & 0 & 0 & 1 & 0 & 0\\
0 & 0 & 0 & 0 & b & 0\\
0 & 0 & 0 & 0 & c & e\\
0 & 0 & 0 & 0 & d & f\\
0 & 0 & 0 & 0 & 0 & g\\
0 & 0 & 0 & 0 & 0 & h
\end{bmatrix}. \en 

Observe that, if $e=0$, then an appropriate unitary transformation involving rows/columns 2,3 only allows to replace $b$ by $\sqrt{b^2+c^2}$ and $c$ by zero. So, there is no need to consider situations where $c\neq 0=e$.  We therefore accept the following

{\sl Convention.} {\em In what follows, whenever $A$ has the form (5.2) and $e=0$ we assume that $c=0$ as well.}

Further simplifications concern the signs of the non-zero entries of $A$. 
\begin{prop} \label{th:2can}  Assume that $A\in \C^{6\times 6}$ is a rank-three  partial isometry with $\defe A\geq 2$ and such that $C(A)$ contains a circle. Then $A$ is circularly equivalent to \eqref{6x61} with 
\eq{elts} b,c,d,e,h\geq 0, g\in\R, \text{ and } f\geq 0 \text{ if } ce=0\text{ and } f\leq 0 \text{ otherwise}, \en 
\eq{length}b^2+c^2+d^2=1, \quad e^2+f^2+g^2+h^2=1.\en 
\end{prop}
\begin{proof}
With $a=0$, \eqref{6x60} becomes \eqref{6x61}, while non-negativity of $h$ can be achieved by an appropriate rotation. The non-negativity of $b,c,d,e$ can be arranged after that along the same lines as in the proof of Proposition~\ref{two_eigenvalues_case_pro}.

Moving on, from the orthogonality of the columns of $A$:
\eq{cdef} ce+df=0, \en 
and, since $ce,d\geq 0$, either $ce=0$ or $d>0, f\leq 0$. In the former case, a diagonal unitary similarity can be used to adjst the argument of $f$ so it becomes real (or even non-negative, if we wish) preserving the non-negativity of $b,c,d,e$. Either way, the requirement $f\in\R$ can be met.

We now invoke the Kippenhahn polynomial again, this time observing that for the matrix \eqref{6x61} $A_{01}=-cer\im g/16$. So, either $ce=0$ or $g$ is real. In the latter case, the proof is complete.
In the former case, due to Convention $c=0$, while from \eqref{cdef} we conclude that also $df=0$. So, we need to consider two possibilities: $c=d=0$ and $c=f=0$. Either way,  there exists a diagonal unitary similarity preserving all the elements of \eqref{6x61} but $g$ and adjusting the argument of the latter so it becomes real. 

For $c=f=0$, let $V=\diag [1,1,,\overline{v},,1,1,v]$, where $\abs{v}=1$ and the argument of $v$ chosen so that $vg\geq 0$. A direct computation shows then that
\[ VAV^* =\begin{bmatrix}
0 & 0 & 0 & 1 & 0 & 0\\
0 & 0 & 0 & 0 & b & 0\\
0 & 0 & 0 & 0 & 0 & e\\
0 & 0 & 0 & 0 & d & 0\\
0 & 0 & 0 & 0 & 0 & vg\\
0 & 0 & 0 & 0 & 0 & h
\end{bmatrix},\]
and we are done. The case $c=d=0$ can be treated similarly.

Finally, \eqref{length} is nothing but the observation that the unit lengths of the columns of $A$ is preserved under circular equivalence. 
\end{proof}

Before proceeding further, let us make a general observation about matrices \eqref{6x61} satisfying \eqref{elts}--\eqref{cdef} not supposing a priori that its Kippenhahn curve contains a circle. Recall that for any matrix $A$ its {\em numerical radius} $w(A)$ is defined as \[ w(A)=\max \{ \abs{z}\colon z\in W(A)\}. \] 

\begin{prop}\label{th:gen6x61} Let $A$ be unitarily similar to \eqref{6x61}. If conditions \eqref{elts}--\eqref{cdef} hold, then:  {\em (i)} $C(A)$ (and therefore $W(A)$) is symmetric about the real axis; {\em (ii)} $w(A)$ is attained at the unique point lying on the real line, or $W(A)$ is a circular disk centered at the origin.  \end{prop} 
\begin{proof} (i) The symmetry of $C(A)$ about the real axis actually holds for any real matrix, and follows easily by taking the adjoint and transposed in \eqref{chpo} and observing that for real matrices $(A^*)^T=A$ while $\theta$ is being replaced by $-\theta$. The respective result for $W(A)$ is well known, and its less obvious refinement can be found, e.g. in \cite{LiTamWu02}.

In our particular setting  a direct computation using  SageMath reveals that a simple  explicit expression for the Kippenhahn polynomial is available, namely: \eq{chpospe}  P(\lambda,\theta)= -\lambda p(\rho)\cos\theta+q(\rho), \en
where $\rho=\lambda^2$ and
\eq{pqrho1}
\begin{cases}
p(\rho)= h \rho^2 - \frac{1}{2} h \rho + \frac{1}{16}((b^2+c^2)h-ceg),  \\
q(\rho)= \rho^3 - \frac{1}{4}(3-h^2)\rho^2 + \frac{1}{16}(c^2+b^2+e^2+1-h^2)\rho - \frac{1}{64} b^2 e^2.    
\end{cases}
\en 
In agreement with (i), \eqref{chpospe} is an even function of $\theta$. Its especially simple form is instrumental in proving the second part of this proposition.

(ii) Let $\theta_0$ be (a priori, one of)  the value(s) of $\theta$ for which the largest root $\lambda_0$ of \eqref{chpospe} attains its maximal value $w(A)$. Set $\rho_0=\lambda_0^2$. If $p(\rho_0)>0$, suppose that $\cos\theta_0\neq 1$.  Then, choosing $\theta_1$ such that $\cos\theta_1>\cos\theta_0$ we have 
$P(\lambda_0,\theta_1)=-\lambda_0p(\rho_0)\cos\theta_1+q(\rho_0)<0$. 
But then $P(\lambda_1,\theta_1)=0$ for some $\lambda_1>\lambda_0$ which is a contradiction with the choice of $\lambda_0$. So, $p(\rho_0)>0$ implies $\cos\theta_0=1$, and the numerical radius of $A$ is attained at the right endpoint of the intersection of $W(A)$ with the real line and only there. 

Similarly, $p(\rho_0)<0$ implies $\cos\theta_0=-1$, and the numerical radius of $A$ is attained at the left endpoint of the intersection of $W(A)$ with the real line and only there. 

In the remaining case $p(\rho_0)=0$ we have $P(\lambda_0,\theta)=q(\rho_0)$ independently of $\theta$. So, $W(A)=\{z\colon\abs{z}\leq\lambda_0\}$, which completes the proof. \end{proof} 
\begin{cor} \label{th:geom_crith_circ}
In the setting of Proposition~\ref{th:gen6x61}, the numerical range of $A$ is circular if and only if its numerical radius is attained for at least two different points on the boundary of $W(A)$.\end{cor}

We now move to criteria for matrices \eqref{6x61} to have circular components in their Kippenhahn cures in terms of their entries. 

For completeness, let us first address a side issue when $C(A)$ contains $\{0\}$, i.e., a degenerate circle $\mathcal C_0$.
\begin{thm}\label{th:C0}Let $A$ be circularly similar to the matrix \eqref{6x61} satisfying \eqref{elts}--\eqref{cdef}. Then $C(A)\supset\{0\}$ if and only if $be=0$.   \end{thm} 
\begin{proof}Without loss of generality, $A$ is equal \eqref{6x61}, and so its Kippenhahn polynomial is given by \eqref{chpospe}. Since the origin is a component of the Kippenhahn curve if and only if the Kippenhahn polynomial is divisible by $\lambda$, the result follows from the formula for $q(\rho)$ in \eqref{pqrho1}.\end{proof}

Note that sufficiency in Theorem~\ref{th:C0} is not surprising at all since under the matrix \eqref{6x61} is permutationally similar to a block diagonal matrix with one of the diagonal blocks being $(0)$ if $b=0$ or $e=0$ (when, by Convention, also $c=0$). The converse, however, is not obvious. 

\begin{thm} \label{th:def2crit}Let  $A\in \C^{6\times 6}$ be a rank-three  partial isometry with $\defe A\geq 2$. Then $C(A)$ contains a (non-degenerate) circle if and only if $A$ is circularly equivalent to the matrix \eqref{6x61} with the entries satisfying \eqref{elts}, and the system
 \eq{pqrho} p(\rho)=0, \quad q(\rho)=0,\en 
where $p,q$ are defined by \eqref{pqrho1}, has solutions $\rho_i\in (0,1)$. 
If this is the case, the  circles in question are centered at the origin, with the radii $r_i=\sqrt{\rho_i}$.
\end{thm}
\begin{proof}Due to Proposition~\ref{th:2can}, we may suppose that $A$ is given by \eqref{6x61}, with the non-zero entries' signs as described there. Also, circular components of $C(A)$, if exist, have to be centered at the origin due to Theorem~\ref{th:GWW63}.  Applying Proposition~\ref{th:critcirc} to the Kippenhahn polynomial \eqref{chpospe}, we see that a circle of radius $r$ is contained in $C(A)$ if and only if \eqref{pqrho} holds with $\rho=r^2$. 
\end{proof} 

\section{Nilpotent rank three partial isometries}\label{s:nilp}
For nilpotent matrices \eqref{6x61}, equalities \eqref{length},\eqref{cdef} yield some not quite trivial identities useful in further considerations. For convenience of reference, we list them in a separate lemma.  
\begin{lem}\label{th:ident} In the setting of Proposition~\ref{th:2can}, if in addition $h=0$, 
\eq{id1} 
b^2+c^2+e^2+d^2g^2=1+b^2e^2, \en 
and, if also $g=0$,
\eq{id2}
b^2e^2+d^2+f^2=1.\en 
\end{lem}
\begin{proof}
To establish \eqref{id1}, observe that
\begin{multline*}
b^2+c^2+e^2+d^2g^2= \text{ [using first equality in \eqref{length}] } \\ b^2+c^2+(b^2+c^2+d^2)e^2+d^2g^2= \\ b^2+c^2+b^2e^2+c^2e^2+d^2e^2+d^2g^2= \text{ [using \eqref{cdef}]  }  \\ 
b^2+c^2+b^2e^2+d^2f^2+d^2e^2+d^2g^2=b^2+c^2+b^2e^2+d^2(f^2+e^2+g^2)= \\
\text{ [using second equality in \eqref{length}] } =1-d^2+b^2e^2+d^2 = 1+b^2e^2.
\end{multline*} 
From \eqref{id1}, provided that $g=0$:
\[  b^2e^2+d^2+f^2=b^2+c^2+e^2-1+d^2+f^2 = (b^2+c^2+d^2-1)+(e^2+f^2), \]
which collapses to one when both equalities in \eqref{length} are invoked. 
\end{proof}

From Theorem~\ref{th:def2crit}, the result for nilpotent matrices \eqref{6x61} follows almost immediately. 
\begin{thm}\label{th:nilp}Let a rank three partial isometry $A$ be nilpotent, i.e., circularly equivalent to \eqref{6x61} with $h=0$. Then $C(A)$ contains a (non-trivial) circular component if and only if $ceg=0$. If this condition holds, then $C(A)$ consists of three concentric circles, one of which degenerates into $\{0\}$ if and only if $be=0$:
\eq{3cir} C(A)=\mathcal C_{r_1}\cup \mathcal C_{r_2}\cup\mathcal C_{r_3}. \en 
Here $r_i=\sqrt{\rho_i}$ while  $\rho_1\leq\rho_2\leq\rho_3$ are the roots of the polynomial 
\[  f(\rho):=\rho^3-\frac{3}{4}\rho^2+\frac{1}{16}(b^2+c^2+e^2+1)\rho-\frac{1}{64}b^2e^2.\]
In particular, $W(A)$ is the circular disk bounded by $\mathcal C_{r_3}$. \end{thm}
\begin{proof} Observe that, when $h=0$, \eqref{pqrho} simplifies to $ceg=0\land f(\rho)=0$. 
So, condition $ceg=0$ is necessary for $C(A)$ to contain at least one  non-trivial circle. On the other hand, if it holds, then the number of circles coincides with the number of real roots of $f$. Due to the nature of Kippenhahn polynomials, $f$ has three real roots (counting multiplicities), all located in $[0,1)$.  Moreover, $r_1=0$ if and only if the constant term of $f$ is zero  which agrees with Theorem~\ref{th:C0}. \end{proof} 
More can be said about the values of $r_i$ in \eqref{3cir}. 
\begin{prop} \label{th:rii} Let $A$ be circularly equivalent to the matrix \eqref{6x61} with $h=ceg=0$. Then {\em (i)} $r_1\leq 1/2\leq r_3$, {\em (ii)} $r_2=1/2$ if and only if $g=0$ or $c=d=0$, {\em (iii)} the circles $\mathcal C_{r_i}$ are either pairwise distinct, or all coincide with $\mathcal C_{1/2}$.\end{prop}
\begin{proof} (i) Due to the Vieta theorem, $r_1^3\leq r_1r_2r_3=be/8\leq 1/8$. So, $r_1\leq 1/2$. 
On the other hand, $r_3$ is the numerical radius of $A$ which cannot be smaller than $\norm{A}/2=1/2$.

\noindent
(ii)  Observe that 
$64f(1/4)=b^2+c^2+e^2-1-b^2e^2$, 
and so, according to \eqref{id1}, \eq{64f} f(1/4)=-d^2g^2/64 \en  equals zero
if and only if $g=0$ or $d=0$. If in the latter case $g\neq 0$, then from $ceg=0$ it follows that $ce=0$. Adopting the convention mentioned in the end of the previous section, this implies $c=0$. 

\noindent
(iii) In order for any two of the radii $r_i$ to coincide it is necessary and sufficient that $f$ has a common root with its derivative $f'$. Dividing $f$ by $f'$ we see that this is happening if and only if the root in common is $\rho=1/4$ while $b=e=1$, $c=d=f=g=0$. The matrix $A$ is then nothing but 
\[ \begin{bmatrix} 0 & 0 & 0 & 1 & 0 & 0 \\ 0 & 0 & 0 & 0 & 1 & 0 \\ 0 & 0 & 0 & 0 & 0 & 1\\ 0 & 0 & 0 & 0 & 0 & 0 \\ 0 & 0 & 0 & 0 & 0 & 0 \\ 0 & 0 & 0 & 0 & 0 & 0 \end{bmatrix}, \]
which is permutationally similar to $J_2\oplus J_2\oplus J_2$. It remains to recall that $C(J_2)=\mathcal C_{1/2}$. \end{proof}
Note that according to \eqref{64f} $f(1/4)\leq 0$ which provides an alternative proof for inequality $r_3\geq 1/2$ in (i).

Naturally, $r_i$ are easily computable when $be=0$ since then $f$ is divisible by $\rho$, with the quotient being a quadratic polynomial.
\begin{cor}\label{th:be0}Let in \eqref{6x61} $h=be=ceg=0$. Then \eqref{3cir} holds with $r_1=0$ and $r_{2,3}=\frac{1}{2}\sqrt{\frac{3\pm\sqrt{5-4(b^2+e^2+c^2)}}{2}}$.\end{cor}
 Repeated application of Proposition~\ref{th:uirr} reveals that matrices \eqref{6x61} with $c=0, bdeg\neq 0$ are unitarily irreducible, matrices with $e$ or $b$ (but not both) equal to zero and $dg\neq 0$ are unitarily similar to 
\[ (0)\oplus \begin{bmatrix}0 & 0 & 1 & 0 & 0 \\ 0 & 0 & 0 & b & 0 \\0 & 0 & 0 & d & 0 \\ 0 & 0 & 0 & 0 & g \\0 & 0 & 0 & 0 & 0 \end{bmatrix} 
\text{ and }  
 (0)\oplus \begin{bmatrix}0 & 0 & 1 & 0 & 0 \\ 0 & 0 & 0 & c & e \\0 & 0 & 0 & d & f \\ 0 & 0 & 0 & 0 & g \\0 & 0 & 0 & 0 & 0 \end{bmatrix}, \]
respectively, with the unitarily irreducible $5$-by-$5$ summand in both cases.  Note that this kind of $5$-by-$5$ nilpotent matrices was treated in \cite{HeSpitSu}, and our Corollary~\ref{th:be0} is in agreement with \cite[Proposition 6]{HeSpitSu}. 

Finally, if $e=b=0$, $dg\neq 0$, then $A\sim (0)\oplus (0)\oplus J_4$, and $r_{2,3}=\frac{\sqrt{5}\pm 1}{4}$, in agreement with the formula from Corollary~\ref{th:be0}. 

Let us mention for completeness the paper \cite{Mat10} devoted to circularity of $W(A)$ in case when $A\in\C^{n\times n}$ is nilpotent, not necessarily a partial isometry, but $n\leq 5$.

Due to Corollary~\ref{th:1/2}, conditions (ii) of Proposition~\ref{th:rii} imply that $A$ is unitarily reducible to a block diagonal matrix with $J_2$ as one of its blocks. In the next section, we will address this situation in more detail. 

\section{Matrices with Kippenhahn curves containing $\mathcal C_{1/2}$}
\label{s:C1/2}
Moving forward, we continue to suppose that $A$ is in the form \eqref{6x61}, with conditions \eqref{elts}--\eqref{cdef} satisfied, but do not impose any other a priori conditions on $h$.  

Our next step is to establish the criterion for $C(A)$ to contain $\mathcal C_{1/2}$  alternative to Corollary~\ref{th:1/2} and stated explicitly in terms of the entries of $A$. 
\begin{thm}\label{th:r1/2v1}For $A$ as just described, its Kippenhahn curve contains $\mathcal C_{1/2}$ if and only if either {\em (i)} $c=d=0$ or \em {(ii)} $g=h=0$. \end{thm}
\begin{proof}
Due to already proven statement (ii) of Proposition~\ref{th:rii}, it suffices to consider the case $h\neq 0$ and show that then $c=d=0$ is the criterion for $C(A)$ to contain $\mathcal C_{1/2}$. 

Plugging $\rho=1/4$ into \eqref{pqrho} while taking into consideration the first equality in \eqref{length} yields
\begin{numcases}{}
d^2h+ceg=0,  \label{pq141} \\
 e^2(c^2+d^2)-d^2=0. \label{pq142}
\end{numcases}
So, we just need to show that the system \eqref{pq141}--\eqref{pq142} has only the trivial solution $c=d=0$. 

Suppose $d\neq 0$.  Then $e\neq 0$ due to \eqref{pq141}. From \eqref{cdef} we therefore conclude that $c^2=d^2f^2/e^2$ and then we can rewrite \eqref{pq142} as 
\[ e^2\left(d^2+\frac{d^2f^2}{e^2}\right)=d^2. \]
From here, $e^2+f^2=1$, and so the second formula in \eqref{length} implies $g=h=0$. This is a contradiction, and so in fact $d=0$. But then \eqref{pq142} boils down to $ec=0$ and, according to the Convention, $c=0$. \end{proof} 

According to Corollary~\ref{th:1/2}, matrices satisfying conditions of Theorem~\ref{th:r1/2v1} are unitarily reducible to a matrix with $J_2$ as one of its diagonal blocks. To describe this unitary similarity explicitly, we need to introduce a two-parameter family of matrices
\eq{Ath} A_{t,h}:=\begin{bmatrix} 0 & 0 & 1 & 0 \\ 0 & 0 & 0 & t \\ 0 & 0 & 0 & \sqrt{1-t^2-h^2} \\
0 & 0 & 0 & h \end{bmatrix}, \quad t,h\geq 0, t^2+h^2\leq 1.  \en 
\begin{prop}\label{th:Athred}Let $A$ be of the form \eqref{6x61} satisfying conditions of Theorem~\ref{th:r1/2v1}. Then it is unitarily similar to $J_2\oplus A_{be,h}$. \end{prop}
Note that $A_{be,h}$ actually is $A_{e,h}$ in case (i) and $A_{be,0}$ in case (ii). 
\begin{proof}In case (i) a block diagonal unitary similarity concentrated in (4,5) rows/columns can be used to replace $g\rightarrow \sqrt{f^2+g^2}, \ f\rightarrow 0$. Then another unitary similarity, this time concentrated in (1,2) rows/columns, restores the 4,5 columns of $A$ without changing any other entries. The resulting matrix after the permutational similarity generated by the $(234)\rightarrow (423)$ cycle yields $J_2\oplus A_{e,h}$.

Moving to case (ii), note first of all that we may exclude the situation $f=d=0$. Indeed, due to \eqref{cdef} this would imply $ce=0$ and then, by convention, $c=0$, which is covered by case (i). 

 Supposing $d^2+f^2\neq 0$, let us introduce 
\[ U= \frac{1}{\sqrt{d^2+f^2}} \begin{bmatrix} 0 & 0 & \sqrt{d^2+f^2} & 0 & 0 & 0\\
bf & 0 & 0 & de-cf & 0 & 0\\
cf-de & 0 & 0 & bf & 0 & 0\\
0 & 0 & 0 & 0 & \sqrt{d^2+f^2} & 0\\
0 & f & 0 & 0 & 0 & d\\
0 & -d & 0 & 0 & 0 & f\end{bmatrix}. \]
Clearly, $U$ is unitary. Direct computations, involving \eqref{id2} among other things, show that $U^*AU=J_2\oplus A_{be,0}$. 
\end{proof} 
\begin{cor} \label{th:Kip1}In the setting of Proposition~\ref{th:Athred}, $C(A)=\mathcal C_{1/2}\cup C(A_{be,h})$.  \end{cor} 
Due to Proposition~\ref{th:Athred} and Corollary~\ref{th:Kip1}, further unitary (ir)reducibility of matrices \eqref{Ath} is  of interest. 
\begin{prop}
\label{th:Athredfur}Matrices \eqref{Ath} are unitarily reducible if (at least) one of the conditions $t=0,\ t=1$, or $t^2+h^2=1$ holds, and unitarily irreducible otherwise.  \end{prop} 
\begin{proof} It is easy to see that $A_{0,h}$ is permutationally similar to $(0)\oplus B$, where 
\eq{B} B= \begin{bmatrix} 0 & 1 & 0 \\ 0 & 0 & \sqrt{1-h^2} \\ 0 & 0 & h\end{bmatrix}. \en
This matrix is nothing but $J_3$ if $h=0$ and $J_2\oplus (1) $ if $h=1$. On the other hand, for $0<h<1$ the matrix \eqref{B} is unitarily irreducible. 

The matrix $A_{1,0}$ is permutationally similar to $J_2\oplus J_2$, while for $t^2+h^2=1$ we have 
\eq{Ath1} A_{t,h}= \begin{bmatrix} 0 & 0 & 1 & 0\\ 0 & 0 & 0 & t \\ 0 & 0 & 0 & 0 \\ 0 & 0 & 0  & h\end{bmatrix} 
\text{ permutationally similar to } J_2\oplus \begin{bmatrix} 0 & t \\ 0 & h \end{bmatrix}. \en 

In the remaining situation $t\neq 0,1$, $t^2+h^2\neq 1$, observe that in the partition \eqref{BC} of $A_{t,h}$ we have $B=\diag[1,t]$ and $C=\begin{bmatrix} 0 & \sqrt{1-t^2-h^2} \\ 0 & h\end{bmatrix}$. So, $B$ is full rank while $C$ is unitarily irreducible. Unitary irreducibility of $A_{t,h}$ then follows from Proposition~\ref{th:uirr}.\end{proof} 
Proposition~\ref{th:Athredfur} (or rather its proof) combined with Proposition~\ref{th:Athred} allow to draw some further conclusions concerning the structure of Kippenhahn curves for partial isometries satisfying conditions of Theorem~\ref{th:r1/2v1}, on top of the inclusion $C(A)\supset\mathcal C_{1/2}$.
\begin{prop}\label{th:Kipfur} Let $A$ be as in Theorem~\ref{th:r1/2v1}. 
\begin{enumerate}
\item If $c=d=e=0$, then $C(A)=\{0\}\cup \mathcal C_{1/2}\cup C(B)$, where $B$ is given by \eqref{B}. In particular, $C(A)=\{0\}\cup \mathcal C_{1/2}\cup \mathcal C_{\sqrt{2}/2}$ if $h=0$, and $C(A)$ consists of two copies of $\mathcal C_{1/2}$ and the pair of points $\{0,1\}$ if $h=1$.
\item If $c=d=0$, $e^2+h^2=1$, $h\neq 1$, then $C(A)$ consists of two copies of $\mathcal C_{1/2}$ and the ellipse $E$ with the foci $0, h$ and the major axis of length one.
\item If $h=0$, then $C(A)=\mathcal C_{r_1}\cup \mathcal C_{1/2}\cup \mathcal C_{r_3}$, where
$r_{1,3}=\frac{\sqrt{1\mp\sqrt{1-b^2e^2}}}{2}$. 
\end{enumerate}
\end{prop} 

\begin{proof} The only claim requiring an explanation is {\em 3.} We refer here to Corollary~\ref{th:Kip1} in combination with \cite[Proposition 5]{HeSpitSu}, where it was shown (up to notation used) that $C(A_{t,0})=\mathcal C_{\sqrt{1-\sqrt{1-t^2}}/2}\cup \mathcal C_{\sqrt{1+\sqrt{1-t^2}}/2}$.\end{proof} 

We are now able to make a statement on the structure  of $W(A)$ for matrices that meet the conditions of Theorem~\ref{th:r1/2v1}. Of course, our main tool here is the equality \eqref{nuco}.
\begin{thm}\label{th:nrc1/2} Let $A$ satisfy conditions of Theorem~\ref{th:r1/2v1}. Then $W(A)$ is 
\begin{enumerate}
\item the circular disk $\left\{z\colon \abs{z}\leq \frac{\sqrt{1+\sqrt{1-b^2e^2}}}{2}\right\}$ if $h=0$;
\item $\conv\{\mathcal C_{1/2},1\}$ if $h=1$; 
\item $\conv\{\mathcal C_{1/2}, E\}$  if $e^2+h^2=1, h\neq 0,1$;
\item $W(B)$, where $B$ is given by \eqref{B}, if $e=0, h\neq 0,1$;
\item $W(A_{e,h})$ if $e, h\neq 0,\ e^2+h^2\neq 1$.
\end{enumerate}
\end{thm} 
\begin{proof} Claims {\em 1}--{\em 3} follow directly from the respective descriptions of $C(A)$ in Proposition~\ref{th:Kipfur}. For claims {\em 4}, {\em 5} we only need to observe, in addition to Proposition~\ref{th:Kipfur}'s statements, that both matrices $B$ and $A_{e,h}$ contain $J_2$ as their principal submatrices. Consequently, $\mathcal C_{1/2}$ is absorbed by them. \end{proof}
Several additional comments on the shape of $W(A)$ in cases {\em 2}---{\em 5} are in order. 

Clearly, $W(A)$ is ice cream cone-shaped in case {\em 2}. For {\em 3}, observe that the rightmost point of $E$ is 
$((1+h)/2,0)$, and so it lies outside of $\mathcal C_{1/2}$. The shape of $W(A)$ is somewhat like ice cream cone, but with a ``rounded'' corner point. In case {\em 4} situation, the numerical range of $B$ (and therefore of $A$) has a so called {\em ovular} shape, as can be easily verified invoking tests from \cite{KRS}. 
 Namely, $C(B)$ does not contain an elliptical component since conditions of \cite[Theorem 2.2]{KRS} are not satisfied, and $W(B)$ has no flat portions on the boundary because $B$ is unitarily irreducible and $\re( e^{i\theta} B)$ has three distinct eigenvalues for every $\theta$ (see \cite[Proposition 3.2]{KRS}). With these two options eliminated, the ovular shape is the only one that reemains.
\par Finally, in the case {\em 5} setting, $C(A_{e,h})$ does not contain circular components. Indeed,  since $A_{e,h}$ is unitarily irreducible, Theorem~\ref{thm:mat_poly_general} implies that $\mathcal C_{1/2}$ cannot be a component of its Kippenhahn curve. On the other hand, $\rho=1/4$ is the only solution of \eqref{pqrho}
for the respective matrix $A$, and so no other circle can be a part of $C(A_{e,h})$ either.

From here we immediately obtain the following 
\begin{cor} \label{th:circ}In the setting of Theorem~\ref{th:r1/2v1}, the numerical range of the matrix $A$ is a circular disk if and only if $A$ is nilpotent. \end{cor} 
Let us mention in passing that in cases {\em 4} and {\em 5} the circle $C_{1/2}$ lies in the interior of $W(A)$. This follows from the fact that $B$ and $A_{e,h}$ up to a permutational similarity are tridiagonal matrices with each pair of off-diagonal entries having different absolute values.  From here it follows \cite[Corollary 7]{BS041} that their Kippenhahn polynomials have simple roots, and so Kippenhahn curves consist of disjoint components.

\section{Rank three partial isometries of defect two} \label{s:r3d2}
In this section, we tackle the remaining case, namely, that of partial isometries $A$ with $\defe A=2$ and such that $C(A)$ does not contain $\mathcal C_{1/2}$. According to Theorem~\ref{th:r1/2v1}, in \eqref{6x61} we then have $d^2+c^2>0$. In what follows, it is convenient to denote \eq{x} x:=\sqrt{d^2+\frac{ceg}{h}}.\en 
 Of course, $x$ exists only if $d^2+\frac{ceg}{h}\geq 0$.
In what follows we require also $x\neq 0$, since otherwise $C(A)$ would contain $\mathcal C_{1/2}$.

Theorem~\ref{th:def2crit} remains our main tool.
\begin{prop}\label{th:no1/2}  Let $A$ be as in \eqref{6x61}, satisfying \eqref{elts}--\eqref{cdef} and such that 
$h\neq 0, c^2+d^2>0$. Then $C(A)$ contains a circle, degenerate or not,  if and only if either 
\eq{crith+} cegh-d^2g^2+\left(e^2+h^2+\frac{ceg}{h}-1\right)x=0,\quad x\leq 3\en
or 
\eq{crith-} cegh-d^2g^2-\left(e^2+h^2+\frac{ceg}{h}-1\right)x=0, \quad x\leq 1.\en
We have $\mathcal C_{\sqrt{1+x}/2}\subset C(A)$ or 
$\mathcal C_{\sqrt{1-x}/2}\subset C(A)$, respectively, if \eqref{crith+} or \eqref{crith-} holds. No other circles can be contained in $C(A)$. 
\end{prop}  
\begin{proof}First, we invoke Proposition~\ref{th:gen6x61} to observe that \eqref{crith+} or \eqref{crith-} with $x<1$ is satisfied exactly when $C(A)$ contains a non-degenerate circle. Indeed, taking \eqref{length} into consideration, the first equation in \eqref{pqrho} can be rewritten as 
\[ h\left(\rho-\frac{1}{4}\right)^2 =\frac{1}{16}(d^2h+ceg). \]
So, for $h\neq 0$ the solutions are $\rho_\pm=\frac{1}{4}(1\pm x)$. Consequently, two potential radii of the circles contained in $C(A)$ are $r_\pm=\sqrt{1\pm x}/2$. In turn, the second equation in \eqref{pqrho} holds for $\rho_\pm$ exactly when \ the equality in \eqref{crith+}/ \eqref{crith-}) does. 
This follows from the following chain of equalities: 
$$64q(r_\pm^2)=(e^2+h^2-1+\frac{ceg}{h})(1 \pm x) + (-b^2e^2-(\frac{1}{h}-h)ceg + (1-h^2)(b^2+c^2)) =$$
$$
\pm x(e^2+h^2-1+\frac{ceg}{h}) + e^2(1-b^2)+(h^2-1)+\frac{ceg}{h} - \frac{ceg}{h} + hceg + (1-h^2)(b^2+c^2)= 
$$
$$
e^2(1-b^2) + cegh+ (h^2-1)d^2 \pm (e^2+h^2-1+\frac{ceg}{h})x.
$$
It remains to observe that 

$e^2(1-b^2) -d^2(1-h^2) = e^2c^2+e^2d^2-d^2f^2-d^2g^2-d^2e^2= -d^2g^2$.

In turn, \eqref{crith-} with $x=1$ holds exactly when $be=0$ because 
$$cegh+d^2h^2-d^2(g^2+h^2)-\left(e^2+h^2+\frac{ceg}{h}-1\right)=$$
$$h^2-d^2(h^2+g^2)-(e^2+h^2+1-d^2-1) = $$
$$d^2(1-h^2-g^2)-e^2 = d^2(f^2+e^2)-e^2 = $$
$$c^2e^2+d^2e^2-e^2 = -b^2e^2=0.$$ 
According to Theorem~\ref{th:C0}, this corresponds to the degenerate circle $\{0\}$.
\end{proof}

Things simplify nicely in case $cg=0$.

\begin{thm}\label{th:c0}Let $A$ be as in \eqref{6x61}, satisfying \eqref{elts}--\eqref{cdef} and such that $h\neq 0,\ cg=0$. Then $C(A)$ contains a circle different from $\mathcal C_{1/2}$ if and only if $d\neq 0,\ c=g=0$. If these conditions hold, then 
\eq{2c1e} C(A)=\mathcal C_{\sqrt{1-d}/2}\cup \mathcal C_{\sqrt{1+d}/2}\cup E, \en
where $E$ is the ellipse with the foci $0,h$ and major axis of the length one.
\end{thm} 
\begin{proof}{\sl Necessity.}
If $cg=0$, then $x=d$, and conditions \eqref{crith+}/\eqref{crith-} take the form 
$d^2g^2\pm (f^2+g^2)d=0$. 
By Proposition~\ref{th:Kipfur} we need $d\neq 0$ in order for $C(A)$ to contain circles different from $\mathcal C_{1/2}$. Consequently, \eqref{crith+}/\eqref{crith-} further reduces to \eq{crith1} dg^2\pm (f^2+g^2)=0.\en 
In case of the plus sign, $dg^2+f^2+g^2=0$, implying $f=g=0$. Due to \eqref{cdef} then $ce=0$, and by Convention $c=0$.

In case of the minus sign, $dg^2=f^2+g^2$. This is only possible if $d=1$ and $f=0$. The respective circle degenerates into $\{0\}$. 

{\sl Sufficiency.} Let $c=g=0$. Due to \eqref{cdef}, then also $f=0$,
and \eqref{crith1} holds with both sign choices. If $d<1$, then 
 $\mathcal C_{\sqrt{1\pm d}/2}\subset C(A)$ by Proposition~\ref{th:no1/2}. 
 
 Observe further that the respective matrix $A$ is permutationally similar to 
\eq{g0} \begin{bmatrix}0 &  e\\ 0 & h\end{bmatrix}\oplus B, \text{ where } B=\begin{bmatrix}0 & 0 & 1 & 0\\ 0 & 0 & 0 & b \\ 0 & 0 & 0 & d\\ 0 & 0 & 0 & 0\end{bmatrix}. \en 
Due to the Elliptical range theorem, the Kippenhahn curve of the first summand is exactly $E$.  Consequently, $C(A)=E\cup C(B)$, and so  
$\mathcal C_{\sqrt{1\pm d}/2}\subset C(B)$. Since  $B\in\C^{4\times 4}$, the two circles $\mathcal C_{\sqrt{1\pm d}/2}$ actually constitute the whole $C(B)$, and \eqref{2c1e} follows. 

If $d=1$, then $b=0$, and the summand $B$ in \eqref{g0} by a permutational similarity reduces further to $(0)\oplus J_3$, so \eqref{2c1e} holds again. \end{proof}

Note for completeness that the matrix $B$ in \eqref{g0} is unitarily irreducible when $b\neq 0$.

\begin{cor}\label{th:circgo} The numerical range of $A$ satisfying conditions of Theorem~\ref{th:c0} is the circular disk $\{z\colon \abs{z}\leq \sqrt{1+d}/2\}$ if $d\geq 2h+h^2$ and a (non-circular) convex hull of $E$ and $\mathcal C_{\sqrt{1+d}/2}$otherwise.\end{cor}
This follows immediately from the fact that the numerical radius of $E$ is attained at its rightmost point $((1+h)/2,0)$. 

\begin{thm}\label{th:2circ} Let $A$ be as in Proposition~\ref{th:no1/2}. Then : {\em (i)} $C(A)$ contains two non-degenerate circles different from $\mathcal C_{1/2}$ if and only if, in addition to other conditions imposed on the entries of $A$, also $c=f=g=0$ while $d\neq 0,1$; {\em (ii)} If this is the case, then  \eqref{2c1e} holds.
\end{thm}  
\begin{proof} If $C(A)$ contains two circles different from $\{0\}$ and $\mathcal C_{1/2}$, then Proposition~\ref{th:no1/2} implies that \eq{fg}  cegh-d^2g^2= e^2+h^2+\frac{ceg}{h}-1= 0.\en
From the first equality in \eqref{fg} with the use of \eqref{cdef} we find that
\eq{fg1_new} dfgh+d^2g^2 = dg(fh+dg)=0. \en
From the second equality in \eqref{fg} using \eqref{length} we infer that \eq{fg1} \frac{ceg}{h}=1-e^2-h^2=f^2+g^2.\en  
Plugging this into the first equality of \eqref{fg}, and using \eqref{cdef} one more time:
\eq{fg3} (f^2+g^2)h^2-d^2g^2=0. \en
From \eqref{fg3} we conclude that if either $d=0$ or $g=0$, then $f=g=0$. If $d=0$, then  by Convention and \eqref{cdef} we infer $c=0$ and therefore from \eqref{fg1} $x=f=g=0$. Hence $d\neq0$. Now from $\eqref{fg1_new}$ either $g=0$ or $fh+dg = 0$. If the latter holds, then from \eqref{fg3} we have $g^2h^2 = 0$, thus $g=0$ as well. Further, from \eqref{fg1} we conclude that $f=0$. Finally, using the Convention and \eqref{length} we conclude that $c=0$.  It remains to invoke Theorem~\ref{th:c0}. Note that $d=1$ has to be excluded since for this choice one of the circles in \eqref{2c1e} is degenerate. 
\end{proof} 

It remains to consider matrices \eqref{6x61} satisfying \eqref{length},\eqref{cdef} and such that $c,h,g\neq 0$. By convention then also $e\neq 0$, and thus, by \eqref{cdef}, $df\neq 0$ as well.

 From Proposition~\ref{th:uirr} it easily follows that such matrices are unitarily irreducible if $b\neq 0$ and otherwise are permutationally similar to 
\[ (0)\oplus \begin{bmatrix} 0 & 0 & 1 & 0 & 0 \\ 0 & 0 & 0 & c & e\\

0 & 0 & 0 & d & f\\ 0 & 0 & 0 & 0 & g\\ 0 & 0 &  0 &  0 & h \end{bmatrix}, \]
with the second summand not further unitarily reducible.

According to Proposition~\ref{th:2circ}, in this setting  at most one of the condition \eqref{crith+}, \eqref{crith-} may hold, and so $C(A)$ contains at most one circle. Examples in Section \ref{s:Ex} show that both cases can materialize. We are interested in the situation when this circle (provided that it exists) is actually the boundary of $W(A)$, i.e., when $A$ has a circular numerical range, 
\begin{thm}\label{circone} Let $A$ be  circularly similar to the matrix \eqref{6x61} with $cdefgh\neq 0$. Then $W(A)$ is a circular disk if and only if \eqref{crith+} holds and 
\eq{dercond} 3x^2+2h^2x+h^2+e^2-d^2-1> 4hx\sqrt{1+x}.\en
\end{thm} 
\begin{proof}Condition \eqref{crith-}, if holds, corresponds to a circle of the radius $r_-<1/2$. Since $w(A)\geq 1/2$, it cannot possibly serve as the boundary of $W(A)$. 

On the other hand, if \eqref{crith+} holds, the respective value $\rho_+=(1+x)/4$ is the common root of the polynomials $p,q$ in \eqref{chpospe}, and $r_+=\sqrt{\rho_+}$ is an independent of $\theta$ root of the respective Kippenhahn polynomial which is grater than $\frac{1}{2}$. By Proposition~\ref{th:gen6x61}, $w(A)$ is attained on the real line, and therefore it is sufficient to check that $r_+$ is the largest root of both $P(0, \lambda)$ and $P(\pi, \lambda)$. Things simplify further. Since $p$ is a quadratic polynomial with positive leading coefficient and $r_+^2$ is its largest root, we conclude that $P(0, \lambda) < P(\pi, \lambda)$ for $\lambda > r_+$. Hence, to decide whether $W(A)$ is circular or not, it is enough to check whether $r_+$ is the largest root of $P(0, \lambda)$.

Finally, for $\theta=0$  there are at most two roots greater than $\frac{1}{2}$. Indeed, they correspond to positive eigenvalues of the matrix 
\eq{M} A+A^T-I=\begin{bmatrix}
-1 & 0 & 0 & 1 & 0 & 0\\
0 & -1 & 0 & 0 & b & 0\\
0 & 0 & -1 & 0 & c & e\\
1 & 0 & 0 & -1 & d & f\\
0 & b & c & d & -1 & g\\
0 & 0 & e & f & g & 2h-1
\end{bmatrix}=: M.\en 
The eigenvalues of the left upper 5-by-5 submatrix of $M$ are easily computable, and they are equal to $-1,-1\pm\sqrt{1+d},-1\pm\sqrt{1-d}$. 
So, there are four negative and one positive amongst them. The Cauchy interlacing theorem suggests that $M$ itself cannot have more than two positive eigenvalues, as claimed. 

Now, to determine if $r_+$ is the largest or merely the second largest root of $P(0,\lambda)$ we invoke the test according to which it depends on the sign of $P'(0,r_+)$. A direct computation reveals that 
\[ P'(0,r_+)=-\frac{h}{4}\,x(x+1)
+\frac{\sqrt{1+x}}{16}\Bigl(3x^{2}+2h^{2}x+h^{2}+e^{2}-d^{2}-1\Bigr). \]
So, $ P'(0,r_+)>0$ if and only if \eqref{dercond} holds. This completes the proof. 
\end{proof}

\section{Examples of $C(A)$ for $A$ with $\rk A=3, \defe A=2$}\label{s:Ex}

Figures~\ref{fig:ellipse_sticks_out}--\ref{fig:ellipse_inside} plot Kippenhahn curves of matrices satisfying conditions of Theorem~\ref{th:2circ}, and therefore consist of two concentric circles and an ellipse. This ellipse intersects both circles in Figure~\ref{fig:ellipse_sticks_out}, contains the inner circle while intersecting the outer one in Figure~\ref{fig:ellipse_sticks_less}, and lies between the two circles in Figure~\ref{fig:ellipse_inside}. The numerical range of $A$ is circular in the latter case only, in agreement with Corollary~\ref{th:circgo}.

    \begin{figure}[H]
        \centering
        \includegraphics[width=0.7\linewidth]{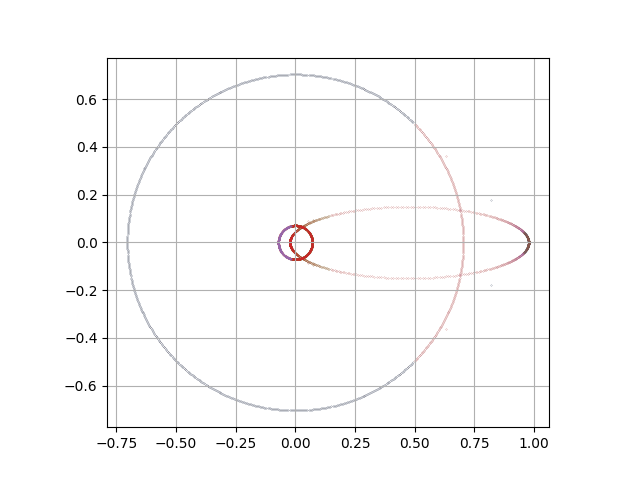}
        \caption{$\sqrt{1-d^2}=0.2$ and $\sqrt{1-h^2}=0.3$}
        \label{fig:ellipse_sticks_out}
    \end{figure}

    \begin{figure}[H]
        \centering
        \includegraphics[width=0.7\linewidth]{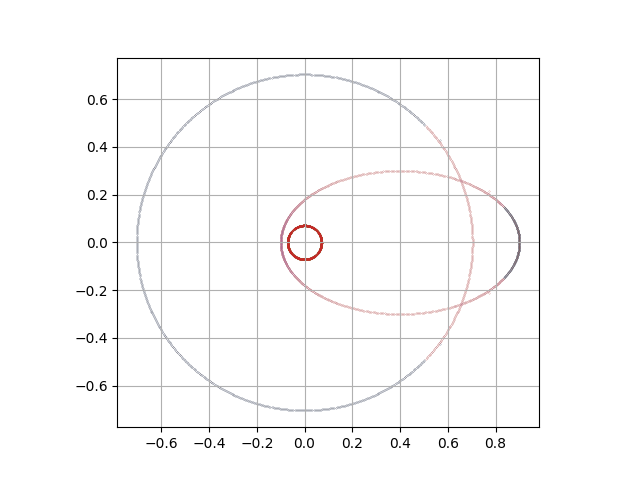}
        \caption{$\sqrt{1-d^2}=0.2$ and $\sqrt{1-h^2}=0.6$}
        \label{fig:ellipse_sticks_less}
    \end{figure}
    \begin{figure} [H]
        \centering
        \includegraphics[width=0.7\linewidth]{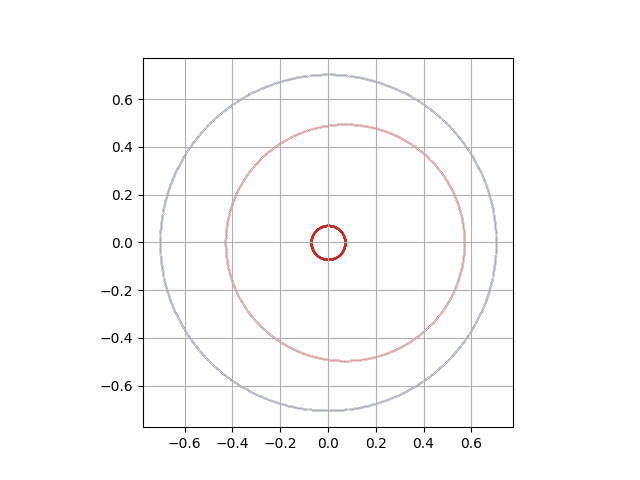}
        \caption{$\sqrt{1-d^2}=0.2$ and $\sqrt{1-h^2}=0.99$}
        \label{fig:ellipse_inside}
    \end{figure}

The rest of examples are for the case of matrices $A$ containing exactly one circle as a component of $C(A)$.

{\em 1.} Set (we rounded the values up to 4 digits) $$A = \begin{pmatrix}
    
0.0 & 0.0 & 0.0 & 1.0 & 0.0 & 0.0 \\
0.0 & 0.0 & 0.0 & 0.0 & 0.9469 & 0.0 \\
0.0 & 0.0 & 0.0 & 0.0 & -0.2926 & 0.1464 \\
0.0 & 0.0 & 0.0 & 0.0 & 0.1327& 0.3228 \\
0.0 & 0.0 & 0.0 & 0.0 & 0.0 & 0.2864 \\
0.0 & 0.0 & 0.0 & 0.0 & 0.0 & 0.8900

\end{pmatrix}.$$
then $C(A)$ contains only circle of radius $r \approx 0.48<0.5$, see Figure~\ref{only_small:fig}.

\begin{figure}[H] 
    \centering
    \includegraphics[width=1\linewidth]{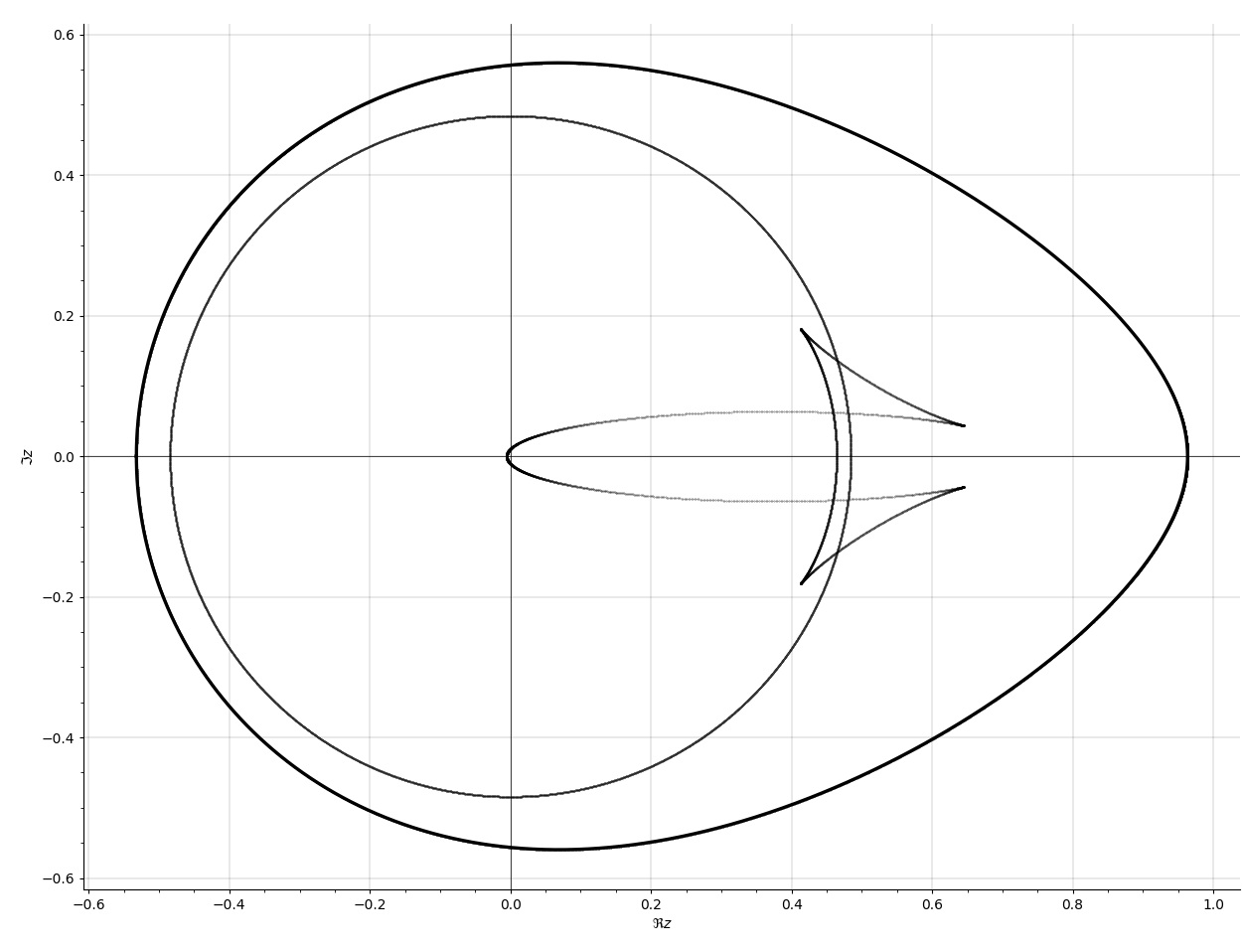}
    \caption{Only circle of radius $r \approx 0.48$}
    \label{only_small:fig}
\end{figure}

{\em 2.} Set (we rounded the values up to 4 digits)
$$A = \begin{pmatrix}

0.0 & 0.0 & 0.0 & 1.0 & 0.0 & 0.0 \\
0.0 & 0.0 & 0.0 & 0.0 & 0.6380 & 0.0 \\
0.0 & 0.0 & 0.0 & 0.0 & 0.3687 & 0.4362 \\
0.0 & 0.0 & 0.0 & 0.0 & 0.6759 & -0.2380 \\
0.0 & 0.0 & 0.0 & 0.0 & 0.0 & -0.7903 \\
0.0 & 0.0 & 0.0 & 0.0 & 0.0 & 0.3583
\end{pmatrix},$$
then there is only one circle of radius $r \approx 0.41$, see Figure~\ref{fig:big_circle}. The Kippenhahn curve consists of three nested closed components, the circle being the intermediate one.
\begin{figure}[H]
    \centering
    \includegraphics[width=1\linewidth]{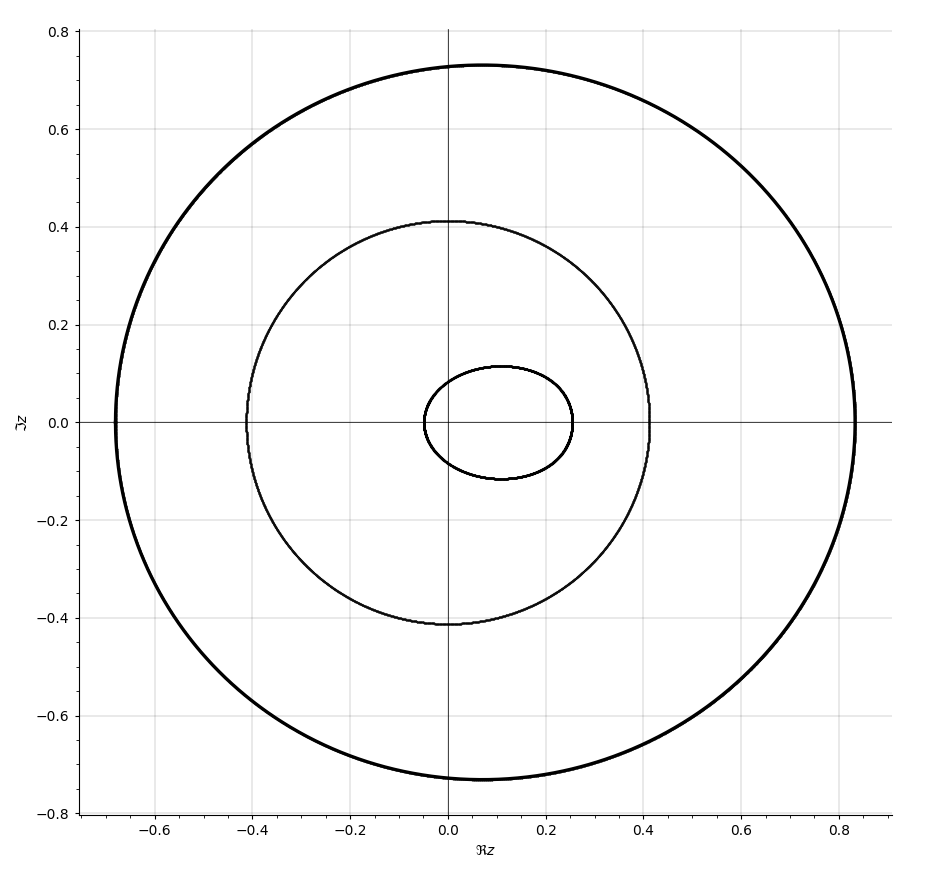}
    \caption{Only circle of radius $r \approx 0.41$.}
    \label{fig:big_circle}
\end{figure}

{\em 3.} Set (we rounded the values up to 4 digits)
$$A = \begin{pmatrix}

0.0 & 0.0 & 0.0 & 1.0 & 0.0 & 0.0 \\
0.0 & 0.0 & 0.0 & 0.0 & 0.0923 & 0.0 \\
0.0 & 0.0 & 0.0 & 0.0 & 0.7491 & 0.3001 \\
0.0 & 0.0 & 0.0 & 0.0 & 0.6558 & -0.3428 \\
0.0 & 0.0 & 0.0 & 0.0 & 0.0 & 0.8725 \\
0.0 & 0.0 & 0.0 & 0.0 & 0.0 & 0.1760
\end{pmatrix},$$
Then there is only circle of radius $r \approx 0.73>0.5$. It is the exterior component of the Kippenhahn curve, and $W(A)$ therefore is a circular disk.

\begin{figure}[H] 
    \centering
    \includegraphics[width=1\linewidth]{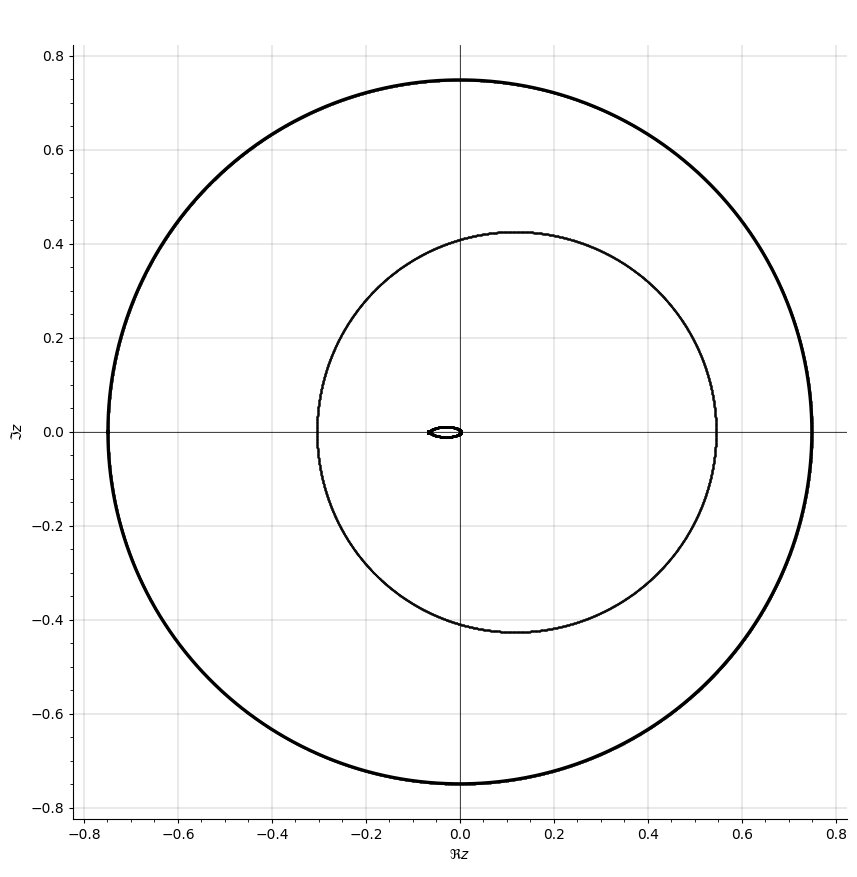}
    \caption{Only circle of radius $r \approx 0.73$}
    \label{only_big:fig}
\end{figure}

\section{Acknowledgements}
The second author is grateful to Timur Garaev for helpful discussions and pointing out Example~\ref{ex:even_dim} in dimension four.
\bibliographystyle{amsplain}
\bibliography{master}
\end{document}